\documentclass{amsart}
\usepackage[utf8]{inputenc}
\usepackage[T1]{fontenc}
\usepackage[english]{babel}

\usepackage{amssymb,amscd}
\usepackage{times}
\usepackage{color}
\usepackage{enumerate}
\usepackage[11pt,curve,matrix,arrow,frame,graph]{xy}
\newcounter{intro}

\newtheorem{theo}[intro]{Theorem}

\newtheorem{thm}{Theorem}[section]
\newtheorem{lem}[thm]{Lemma}
\newtheorem{prop}[thm]{Proposition}
\newtheorem{cor}[thm]{Corollary}

\theoremstyle{remark}
\newtheorem{rem}[thm]{Remark}
\newtheorem{rems}[thm]{Remarks}

\newtheorem*{merci}{Acknowledgements}

\numberwithin{equation}{section}   

\newcounter{counteroman}

\newcommand{\cref}[1]{Corollary~\ref{#1}}

\newcommand{\lref}[1]{Lemma~\ref{#1}}
\newcommand{\pref}[1]{Proposition~\ref{#1}}

\newcommand{\tref}[1]{Theorem~\ref{#1}}

\newcommand{\sref}[1]{Section~\ref{#1}}

\newcommand{\R}{\mathbb{R}}
\newcommand{\C}{\mathbb{C}}

\newcommand{\Z}{\mathbb{Z}}
\newcommand{\bS}{\mathbb{S}}

\newcommand{\bP}{\mathbb{P}}

\newcommand{\mcR}{\mathcal{R}}
\newcommand{\mcW}{\mathcal{W}}\newcommand{\mcZ}{\mathcal{Z}}\newcommand{\mcS}{\mathcal{S}}

\newcommand{\bpm}{\begin{pmatrix}}
\newcommand{\epm}{\end{pmatrix}}

\newcommand{\bdot}{\textbf{.}}
\newcommand{\ps}[2]{\left\langle#1|#2\right\rangle}
\newcommand{\nm}[1]{\left|#1\right|}
\newcommand{\nnm}[1]{\left\|#1\right\|}
\newcommand{\nmr}[1]{\bigl|#1\bigr|}
\newcommand{\nnmr}[1]{\bigl\|#1\bigr\|}
\newcommand{\rst}{\mathring{Ric}}
\newcommand{\rstg}{\mathring{Ric_g}}

\let\ve=\varepsilon
\let\vp=\varphi

\DeclareMathOperator{\ctr}{c}

\DeclareMathOperator{\Id}{Id}

\DeclareMathOperator{\Ric}{Ric}

\DeclareMathOperator{\trace}{tr}
\DeclareMathOperator{\vol}{vol}
\DeclareMathOperator{\Y}{Y}

\def\and{\,\, \mathrm{and}\,\,}

\begin{document}

\title[Optimal integral pinching results]
{Optimal integral pinching results}

\author{Vincent Bour}
\address{Institut Fourier - UMR 5582, Universit\'e J. Fourier, 100 rue des Maths, B.P.~74, 
38402 St Martin d'H\`eres Cedex, France}
\email{Vincent.Bour@ujf-grenoble.fr} 
\author{Gilles Carron}
\address{Laboratoire de Math\'ematiques Jean Leray (UMR 6629), Universit\'e de Nantes, 
2, rue de la Houssini\`ere, B.P.~92208, 44322 Nantes Cedex~3, France}
\email{Gilles.Carron@math.univ-nantes.fr}
\date{\today}
\begin{abstract}
In this article, we generalize the classical Bochner-Weitzenböck theorem for manifolds satisfying an integral pinching on the curvature. We obtain the vanishing of Betti numbers under integral pinching assumptions on the curvature, and characterize the equality case. In particular, we reprove and extend to higher degrees and higher dimensions a number of integral pinching results obtained by M.~Gursky for four-dimensional closed manifolds. 
\end{abstract}
\maketitle

\section{Introduction}

The Bochner method has led to important relations between the topology and the geometry of Riemannian manifolds (see \cite{Ber} for instance). The original theorem of S.~Bochner asserts that a closed $n$-dimensional Riemannian manifold with nonnegative Ricci curvature has a first Betti number smaller than $n$. The technique used by S.~Bochner has been refined, and the result extended to Betti numbers of higher degrees and to various notions of positive curvature. For instance S.~Gallot and D.~Meyer proved in \cite{GM} that the Betti numbers of a closed $n$-dimensional manifold with nonnegative curvature operator must be smaller than those of the torus of dimension $n$. More precisely, they proved that if a closed Riemannian manifold $(M^n,g)$ has a nonnegative curvature operator, i.e. if
\begin{equation}\label{eq:GMpinch}
\rho_g\leq \frac{1}{n(n-1)}R_g,
\end{equation}
where $-\rho_g$ stands for the lowest eigenvalue of the traceless curvature operator and $R_g$ is the scalar curvature of $(M^n,g)$, then for all $1\leq k\leq \frac{n}{2}$,
\begin{itemize}
\item either its $k^{th}$ Betti number $b_k(M^n)$ vanishes,
\item or equality holds in (\ref{eq:GMpinch}), $1\leq b_{k}\leq {n \choose k}$ and every harmonic $k$-form is parallel.
\end{itemize}

Recently, using the Ricci flow, C.~B\"ohm and B.~Wilking proved that a Riemannian manifold with positive curvature operator (i.e. which satisfies the strict inequality in \eqref{eq:GMpinch}) is not only an homological sphere, but is in fact diffeomorphic to a spherical space form (\cite{BW}). A little while later, S.~Brendle and R.~Schoen proved that this is still true for manifolds with $1/4$-pinched sectional curvature (\cite{BS1,BS2}).

In 1998, in his paper \cite{gursky}, M.~Gursky obtained several Bochner's type theorems in dimension four. The striking fact in his work is that the assumption on the curvature is only required in an integral sense. He later refined part of his results in \cite{gursky2}. 

Our formulation of M.~Gursky's results will be given in term of the Yamabe invariant
\begin{equation*}
\Y(M,g):=\inf_{\substack{\varphi\in C_0^\infty(M)\\\vp\neq 0}} \frac{\int_M\left[\frac{4(n-1)}{n-2}|d\varphi|^2+R_g\varphi^{2}\right]dv_g}{
\left(\int_M \varphi^{\frac{2n}{n-2}}dv_g\right)^{\frac{n-2}{n}}}
\end{equation*}
The Yamabe invariant is a conformal invariant: if $u$ is a smooth function, then
\begin{equation*}
\Y(M,g)=\Y\left(M,e^{2u}g\right),
\end{equation*}
hence it only depends on the conformal class $[g]=\left\{e^{2u}g,\,\, u\in C^\infty(M) \right\}$ of the metric $g$. 

When $M$ is closed, the Yamabe invariant has the following geometric interpretation:
\begin{equation*}
\Y(M,[g])=\inf_{\tilde g\in [g]}\left\{\frac{1}{\vol(M,\tilde g)^{1-\frac2n}}\int_M R_{\tilde g} dv_{\tilde g}\right\}\,.
\end{equation*}
According to the work of H.~Yamabe, N.~Trudinger, T.~Aubin and R.~Schoen, we can always find a metric $\tilde g\in [g]$ conformally equivalent to $g$ such that
\begin{equation*}
\frac{1}{\vol(M,\tilde g)^{1-\frac2n}}\int_M R_{\tilde g} dv_{\tilde g}=\Y(M,[g]).
\end{equation*}
The scalar curvature of such a metric $\tilde g$ is constant, and is equal to
\begin{equation*}
R_{\tilde{g}}=\frac{\Y(M^n,[g])}{\vol(M^n,g)^{\frac{2}{n}}}.
\end{equation*}
We call such a metric a \emph{Yamabe minimizer}. Using the Hölder inequality, we see that we always have
\begin{equation*}
 \Y(M^n,[g])\leq \nnm{R_g}_{L^\frac{n}{2}}, 
\end{equation*}
with equality if and only if $g$ is a Yamabe minimizer.

We can state two particular results of M.~Gursky's articles \cite{gursky} and \cite{gursky2} as follows:
\begin{thm}\label{thm:gursky} Assume that $(M^4,g)$ is a closed oriented manifold with positive Yamabe invariant.
\begin{enumerate}[i)]
\item If the traceless part of the Ricci curvature satisfies
\begin{equation}\label{eq:gursky1}
\int_M \nmr{\rstg}^2dv_g\leq \frac{1}{12} \Y(M^4,[g])^2,
\end{equation}
then
\begin{itemize}
\item either its first Betti number $b_1(M^4)$ vanishes,
\item or equality holds in (\ref{eq:gursky1}), $b_1=1$, $g$ is a Yamabe minimizer and $(M^4,g)$ is conformally equivalent to a quotient of $\,\bS^3\times \R$.
\end{itemize}
\item If the Weyl curvature satisfies
\begin{equation}\label{eq:gursky2}
\int_M \nm{W_g}^2dv_g \leq \frac{1}{24} \Y(M^4,[g])^2,
\end{equation}
then
\begin{itemize}
\item either its second Betti number $b_2(M^4)$ vanishes,
\item or equality holds in (\ref{eq:gursky2}), $b_2=1$ and $(M^4,g)$ is conformally equivalent to $\bP^2(\C)$ endowed with the Fubini-Study metric.
\end{itemize}
\end{enumerate}
\end{thm}
The norms of the curvature tensors are taken by considering them as symmetric operators on differential forms, for instance with the Einstein summation convention we have
\begin{equation*}
\nm{W}^2=\frac{1}{4}W_{ijkl}W^{ijkl}\quad \text{and}\quad \nm{Ric}^2=Ric_{ij}Ric^{ij}.
\end{equation*}
M.~Gursky proved these two results by finding a good metric in the conformal class of $g$, for which some pointwise pinching holds. Then, by combining a Bochner-Weitzenböck equation with the pointwise pinching, he was able to prove the vanishing of harmonic forms.  

The purpose of the article is to prove several generalizations of M.~Gursky's result. Instead of trying to obtain a pointwise pinching, we will take advantage of the Sobolev inequality induced by the positivity of the Yamabe invariant. We first prove an integral version of the classical Bochner-Weitzenb\"ock theorem (\tref{thm:bochnerkint}), which allows us to show that a large part of the Bochner theorem of S.~Gallot and D.~Meyer on the Betti numbers of manifolds with nonnegative curvature operator remains true if we only make the assumption in an integral sense:
\begin{theo}\label{thm:Gallot}
If $(M^{n},g)$, $n\geq 4$ is a closed Riemannian manifold such that
\begin{equation}\label{eq:Gallot}
\nnm{\rho_g}_{L^{\frac{n}{2}}}\leq \frac{1}{n(n-1)} \Y(M^n,[g]),
\end{equation}
then for all $1\leq k\leq \frac{n-3}{2}$ or $k=\frac{n}{2}$,
\begin{itemize}
\item either its $k^{th}$ Betti number $b_k(M^n)$ vanishes,
\item or equality holds in \eqref{eq:Gallot} and (up to a conformal change in the case $k=\frac{n}{2}$) the pointwise equality $\rho_g=\frac{1}{n(n-1)}R_g$ holds, $1\leq b_{k}\leq {n \choose k}$, every harmonic $k$-form is parallel and $g$ is a Yamabe minimizer.
\end{itemize}
\end{theo}
\begin{rem}
In \tref{thm:Gallot}, as well as in the other theorems of the article, the two cases are not mutually exclusive, i.e. equality can hold in \eqref{eq:Gallot} while a number of Betti numbers vanish.
\end{rem}
We also obtain an alternative proof of \tref{thm:gursky} based on our integral Bochner-Weitzenböck theorem, and several generalizations of M.~Gursky's result to higher dimensions and higher degrees. In particular, we prove the following extension to higher dimensions of the first part of \tref{thm:gursky}:
\begin{theo}\label{thm:degre1compactnorm}
If $(M^{n},g)$, $n\geq 5$, is a compact Riemannian manifold with positive Yamabe invariant such that 
\begin{equation}\label{eq:degre1compactnorm}
\nnmr{\rstg}_{L^{\frac{n}{2}}}\le \frac{1}{\sqrt{n(n-1)}} \Y(M^n,[g]),
\end{equation}
then
\begin{itemize}
\item either its first Betti number $b_1(M^n)$ vanishes,
\item or equality holds in (\ref{eq:degre1compactnorm}), $b_1=1$, and there exists an Einstein manifold
  $(N^{n-1},h)$ with positive scalar curvature such that $(M^{n},g)$ is isometric to a quotient of the Riemannian product
\begin{equation*}
(N^{n-1}\times \R, h+(dt)^2).
\end{equation*}
\end{itemize}
\end{theo}
\noindent We prove an analogue of the second part of \tref{thm:gursky} in dimension~$6$:
\begin{theo}\label{degre3compact} If $(M^{6},g)$ is a compact Riemannian manifold with positive Yamabe invariant such that 
\begin{equation}\label{eq:degre3compact}
\nnm{W_g}_{L^3}\leq \frac{1}{2\sqrt{10}}\Y(M^6,[g]),
\end{equation}
then
\begin{itemize}
\item either its third Betti number $b_3(M^6)$ vanishes,
\item or equality holds in (\ref{eq:degre3compact}), $b_3=2$, and there exist two positive numbers $a$ and $b$ such that $(M^6,g)$ is conformally equivalent to a quotient of $\left(\bS^3\times \bS^3, a\, g_{\bS^3}+b\, g_{\bS^3}\right)$.
\end{itemize}
\end{theo}
And more generally, we obtain the following result (the constants $a_{n,k}$ and $b_{n,k}$ are defined in \sref{sec:eigenormineq}):
\begin{theo}\label{thm:NormPinch} If $(M^{n},g)$ is a compact Riemannian manifold with positive Yamabe invariant such that for some integer $1\leq k\leq \frac{n}{2}$, $k\neq \frac{n-1}{2}$, the following pinching holds:
\begin{equation}\label{eq:NormPinch}
\left(a_{n,k} \nnm{W}_{\frac{n}{2}}^2 + b_{n,k}\nnmr{\rst}_{\frac{n}{2}}^2\right)^{\frac{1}{2}}\leq\frac{k(n-k)}{n(n-1)} \Y(M,[g]),
\end{equation}
then 
\begin{itemize}
\item either its $k^{th}$ Betti number $b_{k}(M^n)$ vanishes,
\item or $n=4$ and equality holds in \tref{thm:gursky},
\item or $k=1$ and equality holds in \tref{thm:degre1compactnorm},
\item or $k=2$, $n\geq 7$  and $(M^n,g)$ is isometric to a quotient of: $a\left(\bS^2\times \frac{1}{n-5}\bS^{n-2}\right)$.
\item or $k=3$, $n=6$ and equality holds in \tref{degre3compact}.
\end{itemize}
\end{theo}

It should be noticed that for a closed four-dimensional manifold $(M^4,g)$, the condition
\begin{equation*}
\int_M\nm{W_g}^2dv_g+\frac12\int_M\nmr{\rstg}^2dv_g\leq \frac{1}{24} \int_M R_g^2dv_g
\end{equation*}
is conformally invariant, hence choosing $g$ to be a Yamabe minimizer, one can state the following corollary of \tref{thm:gursky}:
\begin{cor}
If $(M^4,g)$ is a closed Riemannian manifold with positive Yamabe invariant such that 
\begin{equation}\label{eq:pinchingchang}
\int_M\nm{W_g}^2dv_g+\frac12\int_M\nmr{\rstg}^2dv_g\leq \frac{1}{24} \int_M R_g^2dv_g, 
\end{equation}
then
\begin{enumerate}[i)]
\item either the Betti numbers $b_1(M^4)$ and $b_2(M^4)$ vanish,
\item or equality holds in \eqref{eq:pinchingchang}, the first Betti number vanishes, $b_2=1$ and $(M^4,g)$ is conformally equivalent to $\bP^2(\C)$ endowed with the Fubini-Study metric (up to orientation),
\item or equality holds in \eqref{eq:pinchingchang}, the second Betti number vanishes, $b_1=1$, and $(M^4,g)$ is conformally equivalent to  a quotient of $\bS^3\times \R$.
\end{enumerate}
\end{cor}
A.~Chang, M.~Gursky and P.~Yang proved in \cite{changall1,changall2} that when the strict equality holds in \eqref{eq:pinchingchang}, the manifold is not only an homological sphere, but is in fact diffeomorphic to a quotient of the round sphere $\bS^4$.

In \cite{bour}, the first author has been able to recover part of this result by using the gradient flow of a quadratic curvature functional. An important step in the proof is to rule out the formation of singularities by a blow-up analysis: if a singularity occurs along the flow, the curvature must blow up, and one can consider a sequence of metrics near the singular time with curvature going to infinity. After a suitable dilatation, this sequence actually converges to a ``singularity model'', which is a non-compact manifold satisfying an integral pinching condition. The classification of the  singularities of those flows is therefore directly related to integral pinching results on non-compact manifolds. In \sref{sec:noncompact}, we will prove the following extension of \tref{thm:degre1compactnorm} to non-compact manifolds:
\begin{theo}\label{thm:noncompact}
Let $(M^{n},g)$, $n \geq 4$, be a complete non-compact Riemannian manifold with positive Yamabe invariant. Assume that the lowest eigenvalue of the Ricci curvature satisfies  $Ric_- \in L^{p}$ for some $p >\frac{n}{2}$, and assume that $R_g\in L^{\frac n2}$. If
\begin{equation}\label{eq:degre1complet}
  \nnmr{\rstg}_{L^{\frac{n}{2}}}+\frac{n-4}{4\sqrt{n(n-1)}}\nnm{R_g}_{L^{\frac{n}{2}}} \leq \frac{n}{4}\frac 1{\sqrt{n(n-1)}} \Y(M^n,[g]),
 \end{equation}
 then
 \begin{itemize}
 \item either $H^1_c(M,\Z)=\{0\}$ and in particular $M$ has only one end,
 \item or equality holds in (\ref{eq:degre1complet}), and there exists an Einstein manifold $(N^{n-1},h)$ with positive scalar curvature and $\alpha>0$  such that $(M^n,g)$ or one of its two-fold coverings is isometric to
\begin{equation*}
\left(N^{n-1}\times\mathbb R,\, \alpha\,\cosh^2(t)\left(h+(dt)^2\right)\right).
\end{equation*}
\end{itemize}
\end{theo}

The structure of the paper is the following: in the next section, we recall the Bochner-Weitzenböck formula and state the extension of the Bochner-Weitzenböck theorem to manifolds with an integral pinched curvature. In section 3, we give estimates on the lowest eigenvalue of the traceless Bochner-Weitzenböck curvature and analyze the equality case. In section 4, we define a modified Yamabe invariant and prove a number of results related to the Yamabe invariant. In section 5, we recall the refined Kato inequality for harmonic forms and its equality case for $1$-forms. In section 6, we prove the extended Bochner-Weitzenböck theorem. In section 7, we prove the other optimal integral pinching theorems, and in the last section, we deal with non-compact manifolds and prove \tref{thm:noncompact}. 

\begin{merci} The authors thank Z.~Djadli for his comments on the article, and are partially supported by the grants ACG: ANR-10-BLAN 0105 and FOG: ANR-07-BLAN-0251-01.
\end{merci}

\section{The Bochner-Weitzenb\"ock formula}

We recall that an harmonic $k$-form $\xi$ satisfies the Bochner-Weitzenb\"ock formula
\begin{equation*}
\langle \nabla^*\nabla \xi,\xi\rangle=-\langle\mcR_k \xi,\xi \rangle
\end{equation*}
where the  Bochner-Weitzenb\"ock curvature:
\begin{equation*}
\mcR_k(x)\colon \Lambda^k T^*_xM\rightarrow \Lambda^k T^*_xM
\end{equation*}
is a symmetric operator that can be expressed by using the curvature operator. The trace of $\mcR_k$ is given by
\begin{equation*}
\trace\left(\mcR_k\right)=\left(\dim \Lambda^k T^*_xM\right) \frac{k(n-k)}{n(n-1)}\,R_g.
\end{equation*}
We let $-r_k$ be the lowest eigenvalue of the traceless part of the Bochner-Weitzenb\"ock curvature.

Since the nonnegativity of $\mcR_k$ is equivalent to
\begin{equation*}
r_k\leq \frac{k(n-k)}{n(n-1)}R_g,
\end{equation*}
the classical Bochner-Weitzenb\"ock theorem can be stated as follows:
\begin{thm}\label{thm:bochnerk}
Let $(M^{n},g)$, $n\geq 2$, be a compact Riemannian manifold. If
\begin{equation}\label{eq:pointwdegrekcompact}
r_k\leq \frac{k(n-k)}{n(n-1)}R_g,
\end{equation}
then
\begin{itemize}
\item either its $k^{th}$ Betti number $b_k(M^n)$ vanishes,
\item or equality holds in (\ref{eq:pointwdegrekcompact}), $1\leq b_{k}\leq {n \choose k}$ and every harmonic $k$-form is
parallel.
\end{itemize}
\end{thm}
In section 6, we will prove the following integral version of \tref{thm:bochnerk}:
\begin{thm}\label{thm:bochnerkint}
If $(M^{n},g)$, $n\geq 4$, is a compact Riemannian manifold such that for some integer $1\leq k\leq \frac{n-3}{2}$ or $k=\frac{n}{2}$ the following pinching holds:
\begin{equation}\label{eq:degrekcompact}
\nnm{r_k}_{L^{\frac{n}{2}}}\leq \frac{k(n-k)}{n(n-1)} \Y(M,[g]),
\end{equation}
then 
\begin{itemize}
\item either its $k^{th}$ Betti number $b_{k}(M^n)$ vanishes,
\item or equality holds in  \eqref{eq:degrekcompact} and (up to a conformal change in the case $k=\frac{n}{2}$) the  pointwise equality $r_k=\frac{k(n-k)}{n(n-1)}R_g$ holds, $1\leq b_{k}\leq {n \choose k}$, every harmonic $k$-form is parallel and $g$ is a Yamabe minimizer.
\end{itemize}
\end{thm}
According to \cite{GM}, for all $1\leq k \leq n-1$, we have $r_k\leq k(n-k)\rho_g$, thus \tref{thm:Gallot} is a direct consequence of this theorem. 

In dimension four, if we let $w_g^+$ be the largest eigenvalue of the self-dual part $W_g^+$ of the Weyl curvature and $b_2^+$ be the dimension of the self-dual harmonic $2$-forms, we obtain the following result:
\begin{thm}\label{thm:wplus4D}
If $(M^{4},g)$ is a compact oriented Riemannian manifold such that
\begin{equation}\label{eq:midegree4D}
\nnm{w_g^+}_{L^2}\leq \frac{1}{6} \Y(M^4,[g]),
\end{equation}
then 
\begin{itemize}
\item either  $b_{2}^+(M^4)=0$,
\item or equality holds in (\ref{eq:midegree4D}), $1\leq b_2^+\leq 3$ and for every self-dual harmonic 2-form $\omega$, there is a Yamabe minimizer $\tilde g$ in $[g]$ such that $\omega$  is K\"ahler for $\tilde g$.
\end{itemize}
\end{thm}
Conversely, according to \cite{derdzinski}, for any metric conformally equivalent to one which is Yamabe and K\"ahler, equality holds in (\ref{eq:midegree4D}).

\subsection{Examples of manifolds for which equality holds in (\ref{eq:degrekcompact}).}
Equality holds in \eqref{eq:degrekcompact} for any metric  with nonnegative $\mcR_k$ which is a positive Yamabe minimizer, as soon as $b_k\geq 1$.  According to \cite{GM}, we can construct examples of manifolds with nonnegative $\mcR_k$ by taking products of manifolds with nonnegative curvature operators. According to \cite[IV.2]{BE}, if the product is an Einstein manifold, it will be a Yamabe minimizer.

Let $(M^n,g)$ be a product of round spheres and projective spaces
$$(\mathbb S^{n_1},g_1)\times\cdots\times(\mathbb S^{n_p},g_p)\times(\mathbb \bP^{m_1}(\C),h_1)\times\cdots\times(\bP^{m_q}(\C),h_q),$$
with $n_i\geq 2$. Then $(M,g)$ has a  nonnegative curvature operator. For $(M,g)$ to be Einstein, we have to take $R_{g_i}=\alpha\frac{n_i}{n}$ and $R_{h_i}=\alpha\frac{2m_i}{n}$ for some $\alpha>0$. 

If for some $0\leq p'\leq p$ and $0\leq m'_j\leq m_j$ 
$$\sum_{i=1}^{p'} n_i+2\sum_{j=1}^{q} m'_j=k,$$
then $b_k\geq 1$ and equality holds in (\ref{eq:degrekcompact}). Hence, for all $k\geq 2$, there exist manifolds for which equality holds in (\ref{eq:degrekcompact}).

For $k=1$, according to \cite{schoen}, the quotients of $\bS^{n-1}\times\R$ by a group of transformations generated by isometries of $S^{n-1}$ and a translation of parameter $T>0$ are Yamabe minimizing if and only if $T^2\leq \frac{4\pi^2}{n-2}$. For those manifolds, the equality holds in (\ref{eq:degrekcompact}) and in (\ref{eq:degre1compactnorm}).

\subsection{The Bochner-Weitzenb\"ock curvature in more details.}\label{ssec:BWeq}
The symmetric operator $\mcR_k$ can be seen as a double form of degree $(k,k)$ (see \cite{kulkarni, bourguignon2, Lab05, Lab06}). Using the fact that the curvature can be seen as a double form of degree $(2,2)$, the fact that the metric can be seen as a $(1,1)$-form and the fact that the wedge product induces an algebra structure on the space of double forms (the Kulkarni-Nomizu product), there is a convenient way to write the Bochner-Weitzenb\"ock curvature in degree $k\in [2,n/2]$:
\begin{equation*}
\mcR_k=\mcW_k+\mcZ_k+\mcS_k 
\end{equation*}
with
\begin{equation*}
\mcW_k=-2\frac{g^{k-2}}{(k-2)!}\,\textbf{.}\,W_g,\quad
\mcZ_k=\frac{n-2k}{n-2}\frac{g^{k-1}}{(k-1)!}\,\textbf{.}\,\rstg,\quad
\mcS_k=\frac{k(n-k)}{n(n-1)}R_g\,\Id_{\Lambda^kT^*M}
\end{equation*}
where $\bdot$ is the product on double-forms, $g^j$ is the metric to the power $j$ with respect to this product, $W_g$ is the Weyl curvature, and $\rstg=\Ric_g-\frac{R_g}{n}g$ is the traceless part of the Ricci curvature. On $1$-forms, we have
\begin{equation*}
\mcR_1=\Ric_g=\rstg+\frac{R_g}{n}g.
\end{equation*}
\section{Comparison between the first eigenvalue and the norm of curvature operators}\label{sec:eigenormineq}
In order to obtain estimates on $r_k$, we will use the following lemma:
\begin{lem}\label{lem:eigenendo}
If $A : E\rightarrow E$ is a traceless self-adjoint endomorphism on a Euclidean space $E$ of dimension $d$, then its lowest eigenvalue $a$ satisfies
$$a^2\le \frac{d-1}{d}\nm{A}^2,$$
and equality holds if and only if the spectrum of $A$ is $\{-\nu,\frac{1}{d-1}\nu\}$ with $\nu\geq 0$ and $\frac{1}{d-1}\nu$ of multiplicity $d-1$.
\end{lem}
\begin{proof}
By a simple Lagrange multiplier argument, we see that:
  \begin{equation*}
\left(\inf\Bigl\{\lambda_1\,,\; \sum_{i=1}^d \lambda_i^2=1\;\text{ and }\; \sum_{i=1}^d \lambda_i=0\Bigr\}\right)^2=\frac{d-1}{d}
\end{equation*}
\end{proof}

\noindent For $1\leq k\leq \frac{n-1}{2}$, let define the constants $a_{n,k}$ and $b_{n,k}$ by
\begin{align*}
a_{n,k}&=\left(\binom{n}{k}-1\right)\frac{k(n-k)}{n(n-1)}\frac{4(k-1)(n-k-1)}{(n-2)(n-3)},\\
b_{n,k}&=\left(\binom{n}{k}-1\right)\frac{k(n-k)}{n(n-1)}\frac{(n-2k)^2}{(n-2)^2}.
\end{align*}
\begin{lem}\label{lem:rk}
If $1\leq k\leq \frac{n-1}{2}$, then
\begin{equation*}
r_k^2\le a_{n,k} \nm{W}^2 + b_{n,k}\nmr{\rst}^2
\end{equation*}
and equality holds if and only if there exists a $k$-form $u$ and a real number $\lambda$ such that
$$\mcR_{k}=\lambda\Id- u\otimes u.$$
\end{lem}
\begin{proof}
We apply \lref{lem:eigenendo} to  $\mcW_k+\mcZ_k$, and use the fact that for a traceless operator $T$ on $k$-forms
  \begin{equation*}
    \nm{\frac{g^j}{j!}\textbf{.}T}=\frac{1}{j!}\bigl\langle \ctr^j\frac{g^j}{j!}\textbf{.}T|T\bigr\rangle=\binom{n-2k}{j}\nm{T}^2,
  \end{equation*}
where $\ctr$ is the contraction operator defined in \cite{Lab05}.
\end{proof}
When $k=n/2$ we can refine this inequality by using the fact that the Hodge star operator commutes with $\mcR_{n/2}$ and the fact that the square of the Hodge star operator on $n/2$-forms is $(-1)^{n/2}\Id$.

Let $\Lambda_\pm^{n/2}T_x^*M$ be the eigenspaces of the Hodge star operator and $\mcR_{\pm,n/2}$ be the restriction of the Bochner-Weitzenb\"ock curvature to $\Lambda_\pm^{n/2}T_x^*M$.

We define
\begin{equation*}
a_{n,n/2}=
\begin{cases}
\frac{n(n-2)}{4(n-1)(n-3)}\left(\tbinom{n}{n/2}-2\right)&\text{if $n/2$ is even}\\
\frac{n(n-2)}{8(n-1)(n-3)}\left(\tbinom{n}{n/2}-2\right)&\text{if $n/2$ is odd}.
\end{cases}
\end{equation*}

\begin{lem} 
\begin{equation}\label{rn2}
 r_{n/2}^2\le a_{n,n/2}\,|W|^2
\end{equation}
and equality holds if and only if
\begin{itemize}
\item when $n/2$ is odd: there exists a $n/2$-form $u$ and a real number $\lambda$ such that
\begin{equation*}
\mcR_{n/2}=\lambda\Id-u\otimes u-*u\otimes *u,
\end{equation*}
\item when $n/2$ is even: there is $\ve\in\{-,+\}$ such that $W^{-\ve}=0$ and there exists a $n/2$-form $u$ such that $*u=\ve u$ and a real number $\lambda$ such that
\begin{equation*}
\mcR_{\ve,n/2}=\lambda\Id-u\otimes u.
\end{equation*}
\end{itemize}
\end{lem}
\begin{proof}
  When $n/2$ is odd, all the eigenspaces of the Bochner-Weitzenb\"ock curvature are stable by the Hodge star operator hence they come with an even multiplicity. And when $n/2$ is even we obtain that $r_{n/2}$ is less than the lowest eigenvalue of $\mcR_{\ve,n/2}$.
\end{proof}

\subsection*{Characterization of the equality case.}
An important feature of the Bochner-Weitzenb\"ock curvature is that it satisfies the first Bianchi identity. Seeing once again $\mcR_k$ as a symmetric operator
\begin{equation*}
\mcR_{k}\colon \Lambda^k T^*_xM\rightarrow \Lambda^k T^*_xM,
\end{equation*}
the first Bianchi identity asserts that if $(\theta_i)_i$ is an orthonormal basis of $(T^*_xM,g)$ then
\begin{equation*}
\forall \alpha\in \Lambda^{k-1} T^*_xM,\quad \sum_{i} \theta_i\wedge\mcR_{k}\left(\theta_i\wedge\alpha\right)=0.
\end{equation*}
We now assume that there exist a real number $\lambda$ and a $k$-form $u\in \Lambda^{k}T_x^*M$ such that
\begin{equation*}
\mcR_{k}=\lambda\Id-u\otimes u.
\end{equation*}
We get that for any orthonormal basis $(e_i)_i$, if we let $(\theta_i)_i$ be its dual basis, then (see \cite{kulkarni})
\begin{equation*}
\sum_i u\wedge \theta_i\otimes e_i\llcorner u=0.
\end{equation*}
We introduce the orthogonal decomposition $T_xM=V\oplus V^\perp$ where
\begin{equation*}
V^\perp=\{v, v\llcorner u=0\},
\end{equation*}
and choose an orthonormal basis $(e_i)_i$ of $T_xM$ diagonalizing the quadratic form
\begin{equation*}
v\mapsto |v\llcorner u|^2,
\end{equation*}
and such that $(e_i)_{1\le i\le\ell}$ is a basis of $V$. Then $\{e_i\llcorner u\}_{1\le i\le\ell}$ is an orthogonal family of $\Lambda^{k-1}T^*M$.

From the identity
\begin{equation*}
\sum_i u\wedge \theta_i\otimes e_i\llcorner u=0,
\end{equation*}
we deduce that $i\in\{1,\hdots,\ell\}\Rightarrow u\wedge \theta_i=0$. Hence $\ell=k$ and
\begin{equation*}
u=|u|\, \theta_1\wedge\dotsb\wedge \theta_k= |u| dv_V.
\end{equation*}
We can go one step further. Indeed if $k\in [2,\frac{n-1}{2}]$, the curvature operator is uniquely determined by the Bochner-Weitzenb\"ock curvature: the components of the curvature operator can be expressed by taking contractions of $\mcR_k$ (see \cite[Theorem~4.4]{Lab06}).

We first see that if $T_xM=V\oplus V^\perp$ and $u=dv_V$ then
$$\ctr(u\otimes u)=*_V g_V,$$
where $g_V$ is the metric on $V$ viewed as a double $(1,1)$-form on $V$ and 
$$*_V\colon \Lambda^{(1,1)}V^*\rightarrow  \Lambda^{(k-1,k-1)}V^*$$
is the Hodge star acting on double forms of $V$. The computations of \cite[theorem 4.4]{Lab06} imply that the traceless part of $\ctr^{k-1}(\mcR_k)$ is proportional to the traceless part of the Ricci curvature, hence the Ricci curvature is a linear combination of $g_V$ and $g_{V^\perp}$; we also get that $\ctr^{k-2}(\mcR_k)$ is a linear combination of $g\bdot Ric$, of $g^2$ and of the curvature operator. Hence in our case, we easily get that there are numbers $\alpha=\alpha(x)$, $\beta=\beta(x)$ and $\gamma=\gamma(x)$ such that the curvature operator at $x$ is
$$\alpha  \frac{g^2_V}{2}+\beta  \frac{g^2_{V^\perp}}{2}+\gamma \frac{g^2}{2}.$$
Hence, using the orthogonal decomposition
$$\Lambda^kT^*\tilde M=\bigoplus_{j=0}^k \Lambda^{k-j}V^*\otimes\Lambda^{j}\left(V^\perp\right)^*,$$
we find that the  eigenvalues of the Bochner-Weitzenb\"ock curvature $\mcR_k$ are
$$\alpha j(k-j)+\beta j(n-k-j)+\gamma k(n-k),$$
with multiplicity $\binom{k}{j}\binom{n-k}{j}$,
 where $j\in \{0,\hdots,k\}$. But the assumption asserts that $\mcR_k$  has only two eigenvalues and that the lowest one has multiplicity $1$. The only possible case is $k=2$ and $\alpha=(n-5)\beta$. Moreover, $\beta\geq 0$, since the lowest eigenvalue of the traceless part of $\mcR_k$ is a negative multiple of $\beta$.  Consequently we have:
\begin{prop}\label{prop:normsplit}
If there is a non-zero $k$-form $u$ such that
  $$\mcR_k(x)=\lambda Id-u\otimes u$$
then $k=2$ and $T_x M$ has an orthogonal decomposition
\begin{equation*}
 T_xM=V\oplus V^\perp,
\end{equation*}
with $V^\perp=\{v, v\llcorner u=0\}$ of codimension $2$. Moreover, $u$ is colinear to the volume form of $V$, and the curvature operator is of the form
$$(n-5)\beta \frac{g^2_V}{2}+\beta\frac{g^2_{V^\perp}}{2}+\gamma \frac{g^2}{2},$$
with $\beta\geq 0$.
\end{prop}

When ${n/2}$ is odd, we use the complex structure given by the Hodge star operator on $n/2$-forms and obtain:
\begin{prop}\label{prop:oddmidegnormsplit}
If ${n/2}$ is odd and if there is a non-zero $n/2$-form $u$ such that
$$\mcR_{n/2}=\lambda \Id-u\otimes u-*u\otimes *u,$$
then $T_x M$ has an orthogonal decomposition 
\begin{equation*}
 T_xM=V\oplus V^\perp,
\end{equation*}
with $V=\{v, v\wedge u=0\}$ and  $V^\perp=\{v, v\llcorner u=0\}$ of dimension $n/2$. 

Moreover, $u$ is colinear to the volume form of $V$ and $*u$ is colinear to the volume form of $V^\perp$.
\end{prop}
Indeed, with the same orthogonal decomposition $T_xM=V\oplus V^\perp$ as before, with 
$$V^\perp=\{v, v\llcorner u=0\},$$
we get that for any vector $w\in V^\perp$,
$$w^\flat\wedge *u=*(w\llcorner u)=0.$$
Hence there is a $(\ell-n/2)$-form $\psi\in \Lambda^{\ell-n/2}V^*$ such that $*u=\psi\wedge dv_{V^\perp}$ and $u=*_V  \psi.$

The Bianchi identity implies that 
$$0=\sum_{i=1}^\ell *_V  \psi\wedge\theta_i\otimes e_i\llcorner *_V  \psi\pm 
\sum_{i=1}^\ell   \psi\wedge dv_{V^\perp} \wedge\theta_i\otimes e_i\llcorner( \psi\wedge dv_{V^\perp}).$$
Because $(e_i\llcorner *_V  \psi)_{1\le i\le\ell}\cup \{e_i\llcorner( \psi\wedge dv_{V^\perp}))_{1\le i\le\ell}$ is an orthogonal
family, we conclude that $\ell=n/2$ and $\psi=1$.

And when ${n/2}$ is even, we have:
\begin{prop}\label{prop:evenmidegnormsplit}
Assume that $n/2$ is even, that for $\ve \in \{-,+\}$ we have $W^{-\ve}=0$ and that there is a non-zero $n/2$-form $u$ such that $*u=\ve u$ and 
\begin{equation*}
\mcR_{\ve ,n/2}=\lambda\Id- u\otimes u.
\end{equation*}
Then $n=4$ and $u$ is colinear to $g(J.,.)$ where $J$ is an unitary complex structure on $T_xM$.
\end{prop}
Indeed, we obtain that the Bianchi operator applied to $u\otimes u$ is a multiple of the Bianchi operator
applied to the Hodge star operator. But the Bianchi operator applied to the Hodge star operator is a multiple of 
the Hodge star operator. Hence, if $(e_i)$ is a orthonormal basis of $T_xM$ then 
$(e_i\llcorner u)_{i}$ is a basis of $\Lambda^{n/2-1}T^*_xM$. 

This can only occur when $n=4$ and when $u=|u| g(J.,.)$, where $J$ is an unitary complex structure on $T_xM$.

\section{The Yamabe invariant}\label{sec:yamabe}

We recall that when $M$ is closed, there always exists a positive smooth function $\vp$ such that
\begin{equation}\label{eq:Yamabe}
\int_M \left[\frac{4(n-1)}{n-2}|d\vp|^2+R_g\vp^{2}\right]dv_g=\Y(M,[g]),\,\, \mathrm{and}\,\, \int_M \vp^{\frac{2n}{n-2}}dv_g=1.
\end{equation}

Moreover, since $C_0^\infty$ is dense in $H_1^2(M)$ (see \cite{Aub98}), the infimum defining the Yamabe invariant can also be taken over $H_1^2(M)$, and any function $\vp\in H_1^2(M)$ with $\nnm{\vp}_{L^{\frac{2n}{n-2}}}=1$ attaining the infimum is smooth, positive and solution to the Yamabe equation:
\begin{equation}\label{eq:YamEq}
 \frac{4(n-1)}{n-2}\Delta_g \vp+R_g \vp = \Y(M,[g])\vp^{\frac{n+2}{n-2}}
\end{equation}
(see \cite{Aub98}). We can also write that equation 
\begin{equation*}
 L_g(\vp)=\Y(M,[g])\vp^{\frac{n+2}{n-2}},
\end{equation*}
where $L_g$ denotes the conformal laplacian
\begin{equation*}
L_g=\frac{4(n-1)}{n-2}\Delta_g+R_g,
\end{equation*}
and satisfies the following conformal covariance property: if $\tilde g=u^{\frac{4}{n-2}}g$, with $u$ a smooth positive function, then
\begin{equation*}
u^{\frac{n+2}{n-2}}L_{\tilde g}(\vp)=L_g(u\vp).
\end{equation*}
It follows in particular that
\begin{equation*}\label{eq:scalarChange}
u^{\frac{n+2}{n-2}}R_{\tilde g}=\frac{4(n-1)}{n-2}\Delta_gu+ R_gu.
\end{equation*}
Therefore, if $u$ is a positive smooth solution of \eqref{eq:YamEq}, then the metric $\tilde{g}=u^{\frac{4}{n-2}} g$ is a Yamabe minimizer.

\subsection{The modified Yamabe invariant}
For $\beta\ge 0$, we introduce the modified Yamabe invariant:

\begin{equation}\label{eq:yamabeModif}
\Y_g(\beta):=\inf_{\substack{\vp\in C_0^\infty(M)\\ \vp\neq 0}} \frac{\int_M \left(\frac{4(n-1)}{n-2}\nm{d\vp}^2+\beta R_g\vp^{2}\right)dv_g}{
\left(\int_M \vp^{\frac{2n}{n-2}}dv_g\right)^{\frac{n-2}{n}}}.
\end{equation}
In particular, for $\beta=1$, this modified Yamabe invariant  is the Yamabe invariant:
$$\Y_g(1)=\Y(M,[g]).$$
The function $\beta\to \Y_g(\beta)$ is an infimum of affine functions of $\beta$, hence it is concave and for all $\beta\in \left[0,1\right]$ we obtain
\begin{equation*}
\left(1-\beta\right) \Y_g(0)+\beta \Y(M,[g])\le \Y_g(\beta).
\end{equation*}
When $(M,g)$ is closed, $\Y_g(0)=0$ and we have:
\begin{prop}\label{prop:EqualityYaModif}
If $(M^n,g)$ is a closed Riemannian manifold, then
\begin{equation}\label{eq:YaModif}
\beta \Y(M,[g])\leq \Y_g(\beta).
\end{equation}
If $\beta\in \left(0,1\right)$, then equality holds in this inequality if and only if $g$ is a Yamabe minimizer, and the only functions attaining the infimum in \eqref{eq:yamabeModif} are constant functions.
\end{prop}
\begin{proof}
Since $\beta\to \Y_g(\beta)$ is concave, it is equal to its chord $\beta \Y(M,[g])$ at an interior point $\beta\in (0,1)$ if and only if it is affine. Then, if for some $u$ and some $\beta \in (0,1)$, equality is attained in \eqref{eq:yamabeModif}, the affine function of $\beta$ corresponding to $u$ is above the function $\Y_g$ and is equal to $\Y_g(\beta)$ at $\beta$, hence it must be equal to $\Y_g$ on $[0,1]$. Therefore the function $u$ realizes the infimum in \eqref{eq:yamabeModif} for all $\beta\in [0,1]$. Taking $\beta=0$ yields $\int_M \nm{d u}^2dv_g=0$, hence $u$ is constant. Then, taking $\beta=1$ shows that $g$ is a Yamabe minimizer.
\end{proof}

\subsection{The Yamabe invariant on complete non-compact manifolds}
We still define the Yamabe invariant by
\begin{equation*}
\Y(M,g):=\inf_{\substack{\vp\in C_0^\infty(M)\\ \vp\neq 0}} \frac{\int_M\left(\frac{4(n-1)}{n-2}\nm{d\vp}^2+R_g\vp^{2}\right)dv_g}{
\left(\int_M \vp^{\frac{2n}{n-2}}dv_g\right)^{\frac{n-2}{n}}}
\end{equation*}
and we still have for any smooth function $u$:
\begin{equation*}
\Y(M,g)=\Y(M,e^{2u}g).
\end{equation*}
The Yamabe functional
\begin{equation*}
\mathcal F_g(\vp)=\frac{\int_M\left(\frac{4(n-1)}{n-2}\nm{d\vp}^2+R_g\vp^{2}\right)dv_g}{
\left(\int_M \vp^{\frac{2n}{n-2}}dv_g\right)^{\frac{n-2}{n}}}
\end{equation*}
is also well-defined when $R_g$ is in $L^{\frac n2}(M,g)$, $\vp$ is in $L^{\frac{2n}{n-2}}(M,g)$ and  $d\vp$ is in $ L^2(M,g)$. Moreover, when $g$ is complete, $C_0^\infty(M)$ is dense in the space
\begin{equation*}
H= \{\vp\in L^{\frac{2n}{n-2}}(M,g),\ \nm{d\vp}\in L^2(M,g)\}.
\end{equation*}
Therefore, we also have
\begin{equation*}
\Y(M,[g])=\inf_H \mathcal F_g,
\end{equation*}
and any function in $H$ with $\nnm{\vp}_{L^{\frac{2n}{n-2}}}=1$ attaining the infimum is a weak solution to the Yamabe equation \eqref{eq:YamEq}. If in addition $\vp$ is in $C^{0,\alpha}$, then by classical regularity theorems $\vp$ is smooth, and by maximum principle it is positive (see \cite{Aub98}). 
  
The following Lemma is inspired by \cite[Proposition 2.3]{carher}:
\begin{lem}\label{lem:nonCompactSobolev}
If $(M,g)$ is a complete non-compact Riemannian manifold with a positive Yamabe constant and scalar curvature in $L^{\frac n2}$, then it has infinite volume and satisfies a Sobolev inequality
\begin{equation}\label{eq:Sobolev}
 \nnm{\vp}_{L^{\frac{2n}{n-2}}}^2\leq C\nnm{d \vp}^2_{L^2},
\end{equation}
for some $C>0$ and for all $\vp\in C_0^\infty(M)$.
\end{lem}
\begin{proof}
Let fix some ball $B(x_0,r)\subset M$. Since $\Y(M,[g])>0$ and by using the Hölder inequality, for any smooth functions with support outside the ball $B(x_0,r)$, we have 
\begin{align*}
  \nnm{\vp}_{L^{\frac{2n}{n-2}}}^2\leq \frac 1{\Y(M,[g])} \left(\frac{4(n-1)}{n-2}\nnm{d\vp}_{L^2}^2+\left(\int_{M\setminus B(x_0,r)} \nm{R_g}^{n/2}dv_g\right)^{\frac 2n} \nnm{\vp}_{L^{\frac{2n}{n-2}}}^2\right).
\end{align*}
Since $R_g$ is in $L^{\frac n2}(M,g)$, we can take $r$ such that 
\begin{equation*}
\left(\int_{M\setminus B(x_0,r)} \nm{R_g}^{n/2}dv_g\right)^{\frac 2n}\leq \frac {\Y(M,[g])}2,
\end{equation*}
and we obtain the Sobolev inequality on $M\setminus B(x_0,r)$.

According to \cite[Lemma 3.2]{Heb96}, there exists a uniform bound from below on the volume of any ball $B(y,1)\subset M\setminus B(x_0,r)$.  Since $(M,g)$ is not compact, we can find a sequence of points $x_k$ in $M\setminus B(x_0,r)$ such that $x_k$ is in $B(x_0,k+1)\smallsetminus B(x_0,k)$. Then the balls $B(x_{3k},1)$ are two by two disjoint, and thus 
\begin{equation*}
\vol_g(M)\geq \sum_{k\geq 0} \vol_g(B(x_{3k},1))=\infty.
\end{equation*}

Therefore, according to \cite[Proposition 2.5]{carronduke}, there exists $C'$ such that the Sobolev inequality \eqref{eq:Sobolev} holds on $M$.
\end{proof}

\begin{prop}\label{prop:YamabeCylin}
If $(N^{n-1},h)$ is an Einstein manifold with scalar curvature
\begin{equation*}
 R_h=(n-2)(n-1),
\end{equation*}
then
\begin{equation*}
\frac{\Y(N\times\mathbb R,[h+ds^2])}{\vol(N,h)^{\frac{2}{n}}}=\frac{\Y(\mathbb S^{n})}{\vol(\mathbb S^{n-1})^{\frac{2}{n}}}.
\end{equation*}
\end{prop}
\proof We note that equality holds if $(N,h)$ is the round sphere $\bS^{n-1}$. If $(N,h)$ is not the round sphere, then by the Bishop-Gromov inequality
$$ \vol(N,h)<\vol(\bS^{n-1}).$$
Moreover the conformal class of $g=h+ds^2$ contains the metric $g_0=\frac 1{\cosh^2(s)}(h+ds^2)$, which is isometric to the spherical suspension of $(N,h)$:
\begin{equation*}
(N\times (0,2\pi),\sin^2(r) h+dr^2).
\end{equation*}
The metric $g_0$ is Einstein, has constant scalar curvature equal to $n(n-1)$ and its volume is 
$$\vol(N,h) \frac{\vol(\bS^n)}{\vol(\bS^{n-1})}.$$
The Yamabe invariant of the cylindrical metric $h+ds^2$ hence satisfies
$$\Y(N\times\R, [h+ds^2])\leq \left(\vol(N,h) \frac{\vol (\bS^n)}{\vol (\bS^{n-1})}\right)^{2/n} n(n-1)<\Y(\bS^n).$$
According to \cite[theorem C]{AB}, there is a bounded smooth function $\varphi$ on $N\times \R$ such that
\begin{equation*}
\int_M \nm{\nabla\vp}^2+R_g\vp^2 dv_g=\Y(M,[g]),
\end{equation*}
and there exist some positive real numbers $c, C, \alpha, \beta$ such that for all $x\in N$ and $s\in\R$
\begin{equation*}
c e^{-\alpha |s|} \le \varphi(x,s)\le C e^{-\beta |s|}.
\end{equation*}
We let $\tilde g=\vp^{\frac{4}{n-2}}(h+ds^2)$ and 
\begin{equation*}
\rho(x,s)=\left(\vp^{\frac{2}{n-2}}(x,s)\cosh(s)\right)^{-1},
\end{equation*}
so that
\begin{equation*}
\tilde g=\rho^{-2}  g_0.
\end{equation*}
We can now run the proof of Obata (see \cite{BE} or \cite{Ob}) and show that since $g_0$ is Einstein, $\tilde g$ must also be Einstein. Indeed, if we note $\gamma_g(X)=L_X g-\frac{2}n( \delta_g X) g$, we have
\begin{equation*}
 \gamma^*_{\tilde g}\left(\frac{1}{\rho}\gamma_{g_0}(\nabla_{g_0}\rho)\right)=\frac{2}{n-2} \gamma^*_{\tilde g}\left(\rst_{\tilde g}\right)=0,
\end{equation*}
hence, with $\vec \nu =\frac{\nabla s}{\nm{\nabla s}_{\tilde g}}$, we get for $T>0$
\begin{align*}
  \int_{N\times [-T,T]}\frac 1{\rho^3} \nm{\gamma_{g_0}(\nabla_{g_0}\rho)}_{\tilde g}^2 dv_{\tilde g}&\leq C\int_{\partial(N\times[-T,T])}\frac 1{\rho}\nm{\gamma_{g_0}(\nabla_{g_0}\rho)}_{\tilde g}\nm{\nabla_{g_0}\rho}_{\tilde g}i_{\vec\nu}(dv_{\tilde g})\\
&\leq 2C\max_{\{-T;T\}}\int_{N}\nm{\gamma_{g_0}(\nabla_{g_0}\rho)}_{g_0}\nm{d\rho}_{g_0}\vp^{2\frac{n-1}{n-2}}dv_h\\
\end{align*}
According to \cite{ACM}, the function $\vp$ is polyhomogeneous, and the boundary term goes to zero as $T$ goes to infinity. 

Then, $\nabla_{g_0}\rho$ is a conformal Killing field. But a conformal vector field of the cylindrical metric is a sum of a conformal vector field on $(N,h)$ and a generator of the translation. If $(N,h)$ is not the round sphere, the conformal  Killing fields of $(N,h)$ are Killing fields, hence the conformal factor must be radial. Then we have (see \cite[Claim 2.13]{AB})
\begin{equation*}
\frac{\Y(N\times\mathbb R,[h+ds^2])}{\vol(N,h)^{\frac{2}{n}}}=\frac{\Y(\mathbb S^{n})}{\vol(\mathbb S^{n-1})^{\frac{2}{n}}}.
\end{equation*}
\endproof

\section{The refined Kato inequality}

The classical Kato inequality asserts that if $\xi$ is a smooth $k$-form on a Riemannian manifold $(M^n,g)$, then 
\begin{equation*}
\left|d|\xi|\right|^2\le \left| \nabla \xi\right|^2.
\end{equation*}
When $\xi$ is moreover assumed to be harmonic, i.e. closed and co-closed
$$d\xi=d^*\xi=0,$$
then for $k\in [0,n/2]$, the Kato inequality can be refined as
\begin{equation}\label{eq:1}
\frac{n+1-k}{n-k}\left|d|\xi|\right|^2\le \left| \nabla \xi\right|^2 ,
\end{equation}
See \cite{bourguignon}, and \cite{branson, kato} for the computation of the refined Kato constant. 

\subsection*{The equality case in the refined Kato inequality for $1$-forms.}

Assume that $(M^n,g)$ is a complete Riemannian manifold, that $\xi\in C^\infty(T^*M)$ is an harmonic $1$-form and that equality holds almost everywhere in the refined Kato inequality:
$$\left|d|\xi|\right|^2=\frac{n-1}{n} \left| \nabla \xi \right|^2.$$
We can locally find a primitive $\Phi$ of $\xi$:
$$d\Phi=\xi.$$ 
Then $\Phi$ is an harmonic function and in this case, the refined Kato inequality is in fact the Yau inequality for harmonic functions (\cite[lemma 2]{yau}). Moreover, passing to the normal covering $\pi\colon \widehat M\rightarrow M$ associated to the kernel of the homomorphism
\begin{equation*}
\gamma\in \pi_1(M)\mapsto \int_\gamma \xi,
\end{equation*}
we have $\pi^*\xi=d\Phi$ for an harmonic function $\Phi\in C^\infty(\widehat{M})$ (see for instance \cite{Ha02}). We will now review the proof of a result of P.~Li and J.~Wang (\cite{liwang}).
\begin{prop}\label{liwang}
Assume that $({M^n},g)$ is a complete Riemannian manifold carrying a non-constant harmonic function $\Phi$ such that almost everywhere\begin{equation*}
\left|d|d\Phi|\right|^2=\frac{n-1}{n} \left| \nabla d\Phi \right|^2.
\end{equation*}
Then there exists a complete Riemannian manifold $(N^{n-1},h)$ such that $({M^n},g)$  is isometric to $ N^{n-1}\times \R$ endowed with a warped product metric $\eta^2(t) h+(dt)^2$. Moreover, there are constants $c_1,c_2$ such that:
$$\Phi(x,t)=c_1+c_2\int_0^t \frac{dr}{\eta(r)^{n-1}}.$$
\end{prop}
\proof We  assume that on $U=\{x\in  M^n,\, d\Phi(x)\not=0\}$ 
$$\left|d|d\Phi|\right|^2=\frac{n-1}{n} \left| \nabla d\Phi\right|^2.$$
Moreover, we add a constant to $\Phi$ such that the set 
\begin{equation*}
N=\{x\in U, \Phi(x)=0\}
\end{equation*}
is not empty.

The equality case in the Yau's inequality implies that there is a function 
$a\colon U\rightarrow \R$ such that if we let $\vec\nu=\frac{\nabla \Phi}{|\nabla \Phi|}$, then in the orthogonal decomposition $T_xM^n=\ker(d\Phi)\oplus\R\,\vec\nu $, we have
$$\nabla d\Phi=\left(\begin{array}{cc}  a \Id& 0 \\ 0 & -(n-1) a\end{array}\right).$$
Therefore, we see that
\begin{equation*}
 \vec u\in \ker(d\Phi) \,\Rightarrow\, d_{\vec u}\left(|d \Phi|^2\right)=0.
\end{equation*}
Hence the length of $d \Phi$ is locally constant on the regular level sets of $\Phi$. Moreover, we have 
\begin{equation*}
 \nabla_{\vec u}(\vec\nu)=\frac 1{\nm{d\Phi}}\left(\nabla_{\vec u}(\nabla\Phi) - \ps{\nabla_{\vec u}(\nabla\Phi)}{\vec \nu}\vec\nu\right),
\end{equation*}
and since $\nabla_{\vec\nu}(\nabla\Phi)$ is in  $\R\,\vec\nu$, $\nabla_{\vec\nu}(\vec\nu)=0$ and the integral curves of the vector field $\vec\nu$ are geodesics. 

We consider the map $E\colon  N\times \R\rightarrow {M^n}$ given by
$$E(x,t)=\exp_x(t\vec\nu(x)).$$
For $x\in N$, and $\vec u\in\vec\nu^\perp$, $\nabla_{\vec u}E(x,t)$ is a Jacobi field along $E(x,t)$, hence is orthogonal to $\vec\nu$ for all $t\in \R$.
Consequently, $\Phi(E(x,t))$ only depends on $t$ and there exists a function $\psi\colon \R\rightarrow \R$ such that 
 $\psi(0)=0$ and for all $(x,t)\in N\times\R$,
 \begin{equation*}
\Phi(E(x,t))=\psi(t).
\end{equation*}
We fix $K\subset N$ a compact subset, and we let $(\alpha,\omega)$ be the maximal open set containing $0$ such that $E\colon K\times (\alpha,\omega)\rightarrow {M^n} $ is a local diffeomorphism. Then, on $K\times (\alpha,\omega)$, we have
\begin{equation*}
E^*\left(\nabla d\Phi\right)=\psi'' dt\otimes dt +\psi' \nabla dt.
\end{equation*}
Consequently, 
\begin{equation*}
\psi''=-(n-1) (a\circ E)\quad \text{and}\quad \psi'\nabla dt=(a\circ E)\left(E^* g\right).
\end{equation*}
Hence $a\circ E$ only depends on $t$, and the hypersurfaces $K\times\{t\}\subset (K\times (\alpha,\omega), E^*g)$  are totally umbilical. Therefore, we get that on $K\times (\alpha,\omega)$,
 $$E^*g= \eta^2(t) h+(dt)^2,$$
with\begin{equation}\label{eq:eta}
a\circ E=\frac{\eta'}{\eta}\quad \text{and}\quad \psi(t)=c\int_0^t \frac{dr}{\eta(r)^{n-1}}.
\end{equation}
Now, if $\omega$ is finite, then for some $x\in K$, $(E^*g)(x,w)$ is not invertible, hence we must have $\lim_{t\to \omega}\eta(t)=0$. According to \eqref{eq:eta}, we also have $\lim_{t\to \omega}\eta'(t)=0$, thus $\eta(t)=o(w-t)$ and
\begin{equation*}
\psi(w)=c\int_0^{\omega}\frac{dr}{\eta(r)^{n-1}}=+\infty,
\end{equation*}
which is not possible. Hence $\omega=+\infty$ and the same argument shows that $\alpha=-\infty$. 

Therefore, $E\colon  N\times\R\rightarrow {M^n}$ is an immersion. Since $d\Phi$ is locally constant on the level sets of $\Phi$, $N$ is a connected component of the closed set $\{x\in M, \phi(x)=0\}$, thus is closed. Then, as $E$ is a local isometry, $E(N\times \R)$ is complete, hence closed in $M^n$, and open, thus $E$ is a surjection.

Moreover, if $E(x,s)=E(y,t)$, then $\psi(s)=\psi(t)$ hence $s=t$, and following the flow of $-\vec \nu$ from $E(x,t)$ or $E(y,t)$ for a time $t$, we see that $x=y$. Therefore, $E$ is also injective.
\endproof

\section{The integral Bochner-Weitzenb\"ock Theorem}\label{sec:BW}

In this section, we prove \tref{thm:bochnerkint} and \tref{thm:wplus4D}. Suppose that $\xi$ is a non-trivial harmonic $k$-form on a complete manifold $(M,g)$, not necessarily compact.

For $\ve >0$ we introduce
$$f_\ve =\sqrt{|\xi|^2+\ve^2}.$$
Elementary computations lead to 
$$f_\ve\Delta f_\ve- |df_\ve|^2=\langle \nabla^*\nabla \xi,\xi\rangle-|\nabla\xi|^2,$$
and for $p>0$ we get
\begin{equation*}
\begin{split}
\Delta f_\ve^p&=pf_\ve^{p-2}\left( f_\ve\Delta f_\ve- (p-1)|df_\ve|^2\right)\\
&=pf_\ve^{p-2}\left( \langle \nabla^*\nabla \xi,\xi\rangle-|\nabla\xi|^2+(2-p)|df_\ve|^2\right)\\
&\le pf_\ve^{p-2}\left( \langle \nabla^*\nabla \xi,\xi\rangle-\frac{n+1-k}{n-k} \left|d|\xi|\right|^2+(2-p)|df_\ve|^2\right).
\end{split}
\end{equation*}
Note that we have
$$|df_\ve|^2=\frac{|\xi|^2\left|d|\xi|\right|^2}{ |\xi|^2+\ve^2}\le \left|d|\xi|\right|^2.$$
Hence, choosing $p=\frac{n-1-k}{n-k}$, we have $p\in (0,2)$ and 
$$2-p=\frac{n+1-k}{n-k}.$$
We obtain
$$\Delta f_\ve^p\le pf_\ve^{p-2}\,\, \langle \nabla^*\nabla \xi,\xi\rangle.$$
According to the Bochner-Weitzenb\"ock formula $\langle \nabla^*\nabla \xi,\xi\rangle=-\langle\mcR_k \xi,\xi \rangle$, we then have
\begin{equation*}
\begin{split}
 \Delta f_\ve^p&\le -pf_\ve^{p-2}\,\, \langle\mcR_k  \xi,\xi\rangle\\
& \le -p \frac{k(n-k)}{n(n-1)}\,R_g f_\ve^{p-2} |\xi|^2 +p\, r_k f_\ve^{p-2} |\xi|^2,
\end{split}
\end{equation*}
and we finally get
\begin{equation}\label{eq:basineq}
\Delta f_\ve^p+ \frac{k(n-1-k)}{n(n-1)}\,R_g f_\ve^{p-2}|\xi|^2\leq \frac{n-1-k}{n-k} \,r_k f_\ve^{p-2} |\xi|^2.
\end{equation}

\subsection{The vanishing result}

\noindent If the manifold $(M,g)$ is closed, by multiplying this inequality by $f_\ve^p$ and integrating over $M$, we obtain
\begin{equation*}
\begin{split}\int_M \left|d\left(f_\ve^p\right)\right|^2dv_g+\frac{k(n-1-k)}{n(n-1)}\,\int_M R_g f_\ve^{2(p-1)} |\xi|^2dv_g \leq\frac{n-1-k}{n-k}\,\int_M r_k f_\ve^{2(p-1)} |\xi|^2dv_g.
\end{split}
\end{equation*}

We define  $v=\nm{\xi}^{\frac{n-1-k}{n-k}}$. Since $f_\ve^{2(p-1)}\nm{\xi}^2\leq \nm{\xi}^{2p}$, by Fatou's Lemma we see that $v$ is in $H_1^2(M)$, and by letting $\ve$ go to zero, we get by Lebesgue's dominated convergence theorem that
\begin{equation*}
\int_M \left[|dv|^2+\frac{k(n-1-k)}{n(n-1)}R_gv^{2}\right]dv_g \leq \frac{n-1-k}{n-k} \int_M r_k v^{2}\, dv_g,
\end{equation*}
When $k\not=\frac{n-1}{2}$, we have
\begin{equation*}
\beta=\frac{4k(n-1-k)}{n(n-2)}\in \left(0,1\right],
\end{equation*}
hence we obtain
\begin{equation*}
\frac{n-2}{4(n-1)}\Y_g(\beta)\nnm{v}^2_{L^{\frac{2n}{n-2}}}\leq \frac{n-1-k}{n-k}  \int_M r_k v^{2}\, dv_g,
\end{equation*}
and according to the Hölder inequality,
\begin{equation}\label{eq:holdercompact}
\frac{n-2}{4(n-1)}\Y_g(\beta)\nnm{v}^2_{L^{\frac{2n}{n-2}}}\leq\frac{n-1-k}{n-k} \nnm{r_k}_{L^{n/2}}\nnm{v}^2_{L^{\frac{2n}{n-2}}}.
\end{equation}
Therefore, either $v$ vanishes on $M$, or 
\begin{equation*}
\frac{n-2}{4(n-1)}\Y_g(\beta)\leq\frac{n-1-k}{n-k} \nnm{r_k}_{L^{n/2}}.
\end{equation*}
In that case, according to \pref{prop:EqualityYaModif}, we obtain
\begin{equation*}
\beta\frac{n-2}{4(n-1)}\Y(M,[g])\leq\frac{n-1-k}{n-k} \nnm{r_k}_{L^{n/2}}.
\end{equation*}
hence
\begin{equation}\label{eq:eigenvanish}
 \nnm{r_k}_{L^{n/2}}\geq \frac{k(n-k)}{n(n-1)}\Y(M,[g]).
\end{equation}

For the middle degree $n/2$ when $n/2$ is even, the Hodge star operator $*$ induces a parallel decomposition  $\Lambda^{\frac{n}{2}}T^*M=\Lambda_+^{\frac{n}{2}}T^*M\oplus \Lambda_-^{\frac{n}{2}}T^*M$. And since the traceless Bochner-Weitzenböck curvature $\mcW_{\frac{n}{2}}$ commutes with $*$, it admits a decomposition
$$\mcW_{\frac{n}{2}}=\mcW_{\frac{n}{2}}^+ \oplus \mcW_{\frac{n}{2}}^-.$$
If $\xi$ is a self-dual form, i.e. if $*\xi=\xi$, and if $-r^+_{n/2}$ is the lowest eigenvalue of $\mcW_{n/2}^+$, we get
\begin{equation}\label{eq:eigenvanishplus}
\nnmr{r^+_{n/2}}_{L^{n/2}}\geq \frac{n}{4(n-1)}\Y(M,[g]).
\end{equation}

\subsection{The equality case}\label{ssec:eigenequality}
We will now characterize the equality case of (\ref{eq:eigenvanish}).

\vspace*{10pt}
\noindent{\bfseries When $\mathbf{1\leq k \leq \frac{n-3}{2}}$.} If equality holds in (\ref{eq:eigenvanish}) and $v$ doesn't vanish, then equality must hold everywhere. In particular, we have $\Y_g(\beta)=\beta \Y(M,[g])$, and the function $v$ attains the infimum in \eqref{eq:YaModif}.  When $k\leq  \frac{n-3}{2}$, since we have 
\begin{equation*}
0<\frac{4k(n-1-k)}{n(n-2)}<1,
\end{equation*}
according to \pref{prop:EqualityYaModif}, $g$ is a Yamabe minimizer and $v$ is constant, hence $\nm{\xi}$ is constant. Furthermore, the equality for the H\"older inequality in (\ref{eq:holdercompact}) implies that $r_k$ is constant, hence
$$r_k=\frac{k(n-k)}{n(n-1)}R_g.$$
By \tref{thm:bochnerk}, every harmonic $k$-form is parallel and $b_k\leq {n \choose k}$.

\vspace*{10pt}
\noindent{\bfseries The middle degree.}
If equality holds in (\ref{eq:eigenvanish}), as $\beta=1$,  $g$ is not necessarily a Yamabe minimizer. However, since $v$ must realize the infimum of the Yamabe functional, the metric $\tilde g=v^{\frac{4}{n-2}}g$ is a Yamabe minimizer. Then, the form $\xi$ is still harmonic for $\tilde g$ but has constant $\tilde g-$length
$$|\xi|_{\tilde g}=1.$$
And since the traceless Bochner-Weitzenböck curvature $\mcW_{\frac{n}{2}}$ only depends on the Weyl curvature, the pinching is conformally invariant and equality also holds for $\tilde g$. Then equality in (\ref{eq:holdercompact}) implies that $r_{n/2}(\tilde g)$ is constant, hence $r_{n/2}(\tilde g)=\frac{n}{4(n-1)}R_{\tilde g}$ and by \tref{thm:bochnerk}, every $\tilde g$-harmonic ${n/2}$-form is $\tilde g$-parallel and $b_{n/2}\leq {n \choose {n/2}}$.

\vspace*{10pt}
\noindent{\bfseries The middle degree in dimension 4.}
If $b_2^+\neq 0$, then according to (\ref{eq:eigenvanishplus}) we have 
\begin{equation}\label{eq:eigenpinch4D}
 \frac{1}{3}\Y(M,[g])\leq  \nnm{r_2^+}_{L^{2}}=2\nnm{w^+}_{L^2}.
\end{equation}

If equality holds in (\ref{eq:eigenpinch4D}) and if there exists a non-trivial self-dual harmonic $2$-form $\xi$, then according to the study of the middle degree case, there is a Yamabe minimizer $\tilde g \in [g]$ such that $\xi$ is $\tilde g$-parallel with $\nm{\xi}_{\tilde g}^2=2$. Then $\xi$ is a K\"ahler form on $(M,\tilde g)$.

\section{Pinching involving the norm of the curvature}

On a closed manifold, according to (\ref{eq:eigenvanish}) and the inequalities of \sref{sec:eigenormineq}, if $b_k\neq 0$, we have

\begin{equation}\label{eq:normvanish}
\left(a_{n,k} \nnm{W}_{\frac{n}{2}}^2 + b_{n,k}\nnmr{\rst}_{\frac{n}{2}}^2\right)^{\frac{1}{2}}\geq \nnm{r_k}_{L^{n/2}}\geq\frac{k(n-k)}{n(n-1)} \Y(M,[g]).
\end{equation}

We will now characterize the equality case in \eqref{eq:normvanish}.

\subsection{For one-forms in dimension greater than 5}

If $b_1\neq 0$ and if
\begin{equation*}
\nnmr{\rst}_{L^{\frac{n}{2}}}=\frac{1}{\sqrt{n(n-1)}} \Y(M,[g]),
\end{equation*}
then according to subsection~\ref{ssec:eigenequality}, $Ric_g$ is nonnegative with $b_1$ zero eigenvalues which correspond to $b_1$ parallel vector fields. According to the de Rham splitting theorem, the universal cover of $(M,g)$ splits as a Riemannian product $(N^{n-b_1}\times\R^{b_1}, h+(dt)^2)$. But according to \lref{lem:rk}, $Ric_g$ has only two distinct eigenvalues, hence $b_1=1$ and $(N,h)$ is Einstein with positive scalar curvature. 

\subsection{For one-forms in dimension 4}

Since there must be equality in (\ref{eq:eigenvanish}), there must also be equality in the Kato inequality. Then according to \pref{liwang}, $M$ has a normal cover  $\widehat{M}=N^{3}\times \R$ with a warped product metric
$$\hat g=\eta^2(t)h+(dt)^2,$$ 
where for some $T>0$, $\eta$ is a $T$-periodic function and the deck transformation group is generated by 
$$\gamma(x,t)=(\phi(x), t+T),$$ 
with $\phi\colon N\rightarrow N$ a $h$-isometry.

We can write that $\hat g$ is isometric to $\tilde g=e^{-2f(s)}(h+ds^2)$. Then,
\begin{equation*}
\rst_{\tilde g}=\rst_{h}+\frac{1}{2}(1-f''-(f')^2)\left(h-3 ds^2\right).
\end{equation*}
Since equality holds in the inequality between the first eigenvalue and the norm of $\rst_{\tilde g}$, we have
\begin{equation*}
\rst_{\tilde g}=r_1e^{-2f}\left(ds^2-\frac{1}{3}h\right),
\end{equation*}
then
\begin{equation*}
\rst_h=\left(r_1e^{-2f}+\frac{1}{2}(1-f''-(f')^2)\right)\left(3ds^2-h\right)
\end{equation*}
and by taking the trace on $TN\subset T\widehat M$, we see that $\rst_h$ must vanish. Thus $(N^3,h)$ is Einstein hence of constant sectional curvature, and $(M,g)$ is conformally equivalent to a quotient of $\bS^{3}\times\R$. We recover \tref{thm:gursky} i).

\begin{rems}
i) If the translation parameter $T$ is too large the the product metric cannot be a Yamabe minimizer.
Indeed the second variation of the Yamabe functional has a negative eigenvalue at the product metric when
$$T^2>\frac{4\pi^2(n-1)}{R_{h}}.$$
Conversely, on $\bS^{n-1}\times \bS^1$, the product metric is a Yamabe minimizer as soon as
\begin{equation}\label{eq:TranslParam}
T^2\leq\frac{4\pi^2}{n-2}
\end{equation}
(cf. \cite{schoen}). Therefore, in dimension $4$, if $b_1\neq 0$, equality holds in \eqref{eq:gursky1} if and only if $(M,g)$ is conformally equivalent to a quotient of $\bS^3\times \R$ with translation parameter satisfying \eqref{eq:TranslParam}. 

ii) If $(M,g)$ satisfies the pinching
\begin{equation*}
\int_M \nmr{\rstg}^2dv_g\leq\frac{1}{12} \int_M R_g^2dv_g
\end{equation*}
which is conformally invariant according to the Gauss-Bonnet formula, we can suppose (up to a conformal change) that $g$ is a Yamabe minimizer and satisfies (\ref{eq:gursky1}).
\end{rems}
 
\subsection{For two-forms in dimension 4}

If we have the equality
\begin{equation*}
\nnm{W_g}_{L^2}=\frac{1}{2\sqrt{6}} \Y(M,[g]),
\end{equation*}
and $b_2\neq 0$, then by taking a two-fold covering if M is not orientable and choosing the right orientation, we have equality in (\ref{eq:eigenpinch4D}), $b_2^+(M)=0$ and $W_g^-=0$. Hence $(M,g)$ is conformally equivalent to a K\"ahler self-dual manifold with constant scalar curvature. According to \cite{bourguignon2, derdzinski}, $(M,g)$ is conformally equivalent to $\bP^2(\C)$ endowed with the Fubini-Study metric, and we recover \tref{thm:gursky} ii).

\subsection{In degree $\mathbf{k\in [2,\frac{n-2}{2}]}$ when $n\geq 7$}
If equality holds in (\ref{eq:normvanish}) and if there exists a non-trivial harmonic $k$-form, then according to \pref{prop:normsplit} we must have $k=2$.

When $n\ge 7$, we have $k\leq \frac{n-3}{2}$, and according to subsection~\ref{ssec:eigenequality}, the metric $g$ is a Yamabe minimizer and $\xi$ is parallel. According to \pref{prop:normsplit}, we obtain a parallel decomposition $T^*M=V\oplus V^\perp$, and the universal cover of $(M,g)$ splits as a Riemannian product
$$\pi\colon\tilde M=X_1\times X_2\rightarrow M$$ 
where $X_1$ has dimension $2$ and $\pi^*\xi$ is colinear to $\lambda dv_{X_1}$. Still from \pref{prop:normsplit}, we see that $\gamma=0$, that $X_2$ has constant positive sectional curvature, that we can normalize to be $1$, and that $X_1$ has constant sectional curvature, hence is a $2$-sphere of curvature $n-5$.

\subsection{For $2$-forms in dimension $6$}
We consider a closed manifold $(M^6,g)$ with $b_2\not=0$ which satisfies
$$\nnm{r_2}_{L^{3}}=\frac{4}{15} \Y(M,[g])$$
and
$$|r_2|^2= a_{6,2} \nm{W}^2 + b_{6,2}\nmr{\rst}^2.$$

In this case, there is an harmonic $2$-form $\xi$ for which equality holds in the refined Kato inequality, and the curvature operator is
$$\beta \left(\frac{g_V^2}{2}+\frac{g_{V^\perp}^2}{2}\right)+\gamma \frac{g^2}{2},$$
where at each point 
$$T_xM=V\oplus V^\perp.$$
Following the computations done in \cite{Carron}, we introduce a local orthonormal frame $(e_1,e_2,e_3,\hdots,e_6)$ and its dual frame $(\theta^1,\hdots ,\theta^6)$, with $V=\mathrm{Vect}(e_1,e_2)$. We can write that
$$d|\xi|=\rho |\xi| \theta^1\quad\text{and}\quad\xi=|\xi|\theta^1\wedge \theta^2.$$
The computation leads to
$$\nabla_{e_1} \xi =\rho\xi,\quad \nabla_{e_2} \xi =0\quad\text{and}\quad \nabla_{e_j} \xi =-\frac{1}{4}\rho|\xi| \theta^j\wedge \theta^2,$$
for $j\geq 3$. Hence, writing $\Omega=\frac{\xi}{|\xi|}$, we obtain $\nabla_{e_1} \Omega =\nabla_{e_2} \Omega =0$,
$$R(e_1,e_2)\Omega=(\beta+\gamma) \Omega\quad\text{and}\quad R(e_1,e_2)\Omega=-\nabla_{[e_1,e_2]}\Omega.$$
This implies that 
$$[e_1,e_2]\in V=\mathrm{Vect}(e_1,e_2).$$
Hence $\nabla_{[e_1,e_2]}\Omega=0$ and thus $ \beta+\gamma=0$. However, the scalar curvature of $g$ is
$$R_g=14\beta+30\gamma=-16\beta.$$
This is not possible, since $\beta\ge 0$ and since we have assumed that the Yamabe invariant of $(M,g)$ is positive.

\subsection{The middle degree}
We consider a closed manifold $(M^n,g)$ with $b_{\frac{n}{2}}\neq 0$ and such that the following equality holds
\begin{equation*}
a_{n,n/2} \left(\int_M |W|^{\frac{n}{2}}dv_g\right)^{\frac{2}{n}}=\frac{n}{4(n-1)} \Y(M,[g]).
\end{equation*}

If $n/2$ is even, by \pref{prop:evenmidegnormsplit} we must have $n=4$ and $(M,g)$ conformally equivalent to $\bP^2(\C)$ endowed with the Fubini-Study metric.

If $n/2$ is odd, then according to the middle degree case in \sref{sec:BW}, up to a conformal change $\tilde g= |\xi|^{\frac4n} g$ on the metric, we can suppose that $g$ is a Yamabe minimizer and that $\xi$ is parallel.

According to \pref{prop:oddmidegnormsplit}, the universal cover of $(M,g)$ splits as a Riemannian product $X_1\times X_2$ where $X_1$ and $X_2$ have dimension $n/2$. Moreover, in the orthogonal decomposition
$$\Lambda^{\frac n2} T^*(X_1\times X_2)=\bigoplus_{j=0}^{\frac n2}\Lambda^{j} T^*X_1\otimes\Lambda^{\frac n2-j} T^* X_2,$$
the Bochner-Weitzenb\"ock curvature has the decomposition
$$\mcR_{\frac{n}{2}}=\sum_{j=0}^{\frac{n}{2}}\left(\mcR^{X_1}_{{\frac{n}{2}}-j}\otimes \Id_{\Lambda^{j}T^*X_2}+ \Id_{\Lambda^{{\frac{n}{2}}-j}T^*X_1}\otimes \mcR_{j}^{X_2}\right).$$
Hence for $j\in \{0,\hdots,{\frac{n}{2}}\}$ , $\mcR^{X_1}_{{\frac{n}{2}}-j}$ and $\mcR_{j}^{X_2}$ are multiple of the identity. 
In particular  $\mcR^{X_1}_{{\frac{n}{2}}-2}$ and  $\mcR^{X_2}_{2}$ are multiple of the identity, and by \cite{Taba} or \cite{Lab06}, this implies that $X_1$ and $X_2$ have constant sectional curvature.

Moreover, the  eigenvalues of $\mcR_{\frac{n}{2}}$ are
$$ j\left({\frac{n}{2}}-j\right)\frac{R_g}{n(n-1)},$$
with multiplicity ${n/2 \choose j}^2$, where $j\in \{0,\hdots,{\frac{n}{2}}\}$. The only possibility to have only two distinct eigenvalues is when $n=6$. Then $X_1$ and $X_2$ are two round spheres.

\begin{rem}
  If $X_1$ and $X_2$ are two round spheres of same radius, then the product is Einstein. According to \cite{BE} it is a Yamabe minimizer, and thus equality really holds in \eqref{eq:degre3compact}.
\end{rem}
\section{The non-compact case}\label{sec:noncompact}

We will prove the following result, which implies \tref{thm:noncompact}:
\begin{thm}
Let $(M^{n},g)$, $n\geq 4$, be a complete non-compact Riemannian manifold with positive Yamabe invariant. Assume that  the lowest eigenvalue of the Ricci curvature satisfies  $Ric_- \in L^{p}$  for some $p >\frac{n}{2}$, and assume that $R_g\in L^{\frac n2}$. If
\begin{equation}\label{eq:degre1completeigen}
  \nnmr{r_1}_{L^{\frac{n}{2}}}+\frac{n-4}{4n}\nnm{R_g}_{L^{\frac{n}{2}}} \leq \frac{1}{4}\Y(M,[g]),
 \end{equation}
 then
 \begin{itemize}
 \item either $H^1_c(M,\Z)=\{0\}$ and in particular $M$ has only one end.
 \item or equality holds in (\ref{eq:degre1completeigen}) and there exists an Einstein manifold $(N^{n-1},h)$ with positive scalar curvature and $\alpha>0$  such that $(M^n,g)$ or one of its two-fold covering is isometric to
\begin{equation*}
\left(N^{n-1}\times\mathbb R,\, \alpha\,\cosh^2(t)\left(h+(dt)^2\right)\right).
\end{equation*}
\end{itemize}
\end{thm}

According to \lref{lem:nonCompactSobolev}, there exists $C$ such that the following Sobolev inequality holds:
\begin{equation}\label{eq:EuclSobolev}
\forall\vp\in C_0^\infty(M)\qquad \nnm{\vp}_{L^{\frac{2n}{n-2}}}^2\leq C\nnm{d \vp}^2_{L^2}.
\end{equation}
Then, according to \cite[Proposition 5.2]{carronpedon}, if $H^1_c(M,\Z)\not=\{0\}$, then $M$ or one of its two-fold covering has at least two ends. 

If $M$ has at least two ends, then according to \cite[Theorem 2]{CSZ}, we can find a compact set $K\subset M$ with
$$M\setminus K=\Omega_-\cup \Omega_+,$$ 
and with both $\Omega_-$ and $\Omega_+$ unbounded, and an harmonic function $\Phi\colon M\mapsto (-1,1)$ such that $d\Phi\in L^2$,
\begin{equation*}
\lim_{\substack{x\to \infty\\\; x\in \Omega_-}}\Phi(x)=-1\quad\text{and}\quad\lim_{\substack{x\to \infty\\\; x\in \Omega_+}}\Phi(x)=1.
\end{equation*}
In particular $\xi:=d\Phi$ is an $L^2$ harmonic $1$-form on $(M,g)$.

If $M$ has only one end and $\check\pi\colon \check M\rightarrow M$ is a two-fold covering of $M$ with at least two ends, then $(\check M, \check\pi^*g)$ satisfies the Sobolev inequality \eqref{eq:EuclSobolev}, and we can find a compact set $K\subset M$ such that $\Omega:=M\setminus K$ is connected and  $\check M\setminus \pi^{-1}(K)=\check\Omega_-\cup\check \Omega_+$ with $\check\Omega_-$ and $\check\Omega_+$ unbounded. Then we can find an harmonic function $\check\Phi\colon \check M\rightarrow (-1,1)$ such that $d\check\Phi\in L^2$,
\begin{equation*}
\lim_{\substack{x\to \infty\\\; x\in \check\Omega_-}}\check\Phi(x)=-1\quad\text{and}\quad\lim_{\substack{x\to \infty\\\; x\in \check\Omega_+}}\check\Phi(x)=1.
\end{equation*}
Moreover this function is unique by maximum principle, hence the image of $\check \Phi$ by a deck transformation of $\check\pi\colon \check M\rightarrow M$ is either $\check\Phi$ or $-\check\Phi$. In particular, the function $|\xi|$=$|d\check \Phi|$ is well defined on $M$ and is in $L^2(M,g)$. 

Furthermore, since  $Ric_-$ is in $L^p$ for some $p>n/2$, and since
$$\Delta |\xi|\le -\Ric_-|\xi|,$$
we get by DeGiorgi-Nash-Moser iterative scheme (see for instance \cite[Theorem B.1]{yang}) that $\xi$ is in $L^\infty$, and that
\begin{equation}\label{limitin}
\lim_{x\to \infty} |\xi|=0.
\end{equation}
In particular, the function $v=\nm{\xi}^{\frac{n-2}{n-1}}$ is in $L^{\frac{2n}{n-2}}(M,g)$.

If $\chi$ is a Lipschitz function with compact support and $h$ is a smooth function, we have the integration by parts formula
\begin{equation*}
\int_M \left|d\left(\chi h\right)\right|^2dv_g =\int_M \left[\left|d\chi\right|^2h^2+\chi^2 h\Delta h\right]dv_g,
\end{equation*}
By multiplying inequality \eqref{eq:basineq} by $\chi^2 f_\ve^{p}$ and integrating over $M$, we obtain:
\begin{multline*}
\int_M\nm{d (\chi f_\ve^{p})}^2dv_g+\frac{n-2}{n(n-1)} \int_M R_g f_\ve^{2p-2}\nm{\xi}^2\chi^2dv_g\leq p\int_M r_1\,f_\ve^{2p-2}\nm{\xi}^2\chi^2dv_g\\
+\int_M \nm{d\chi}^2f_{\ve}^{2p}dv_g,
\end{multline*}
where $p=\frac{n-2}{n-1}$. If we let $\ve$ go to zero, we get by Fatou's Lemma and Lebesgue's dominated convergence theorem that
\begin{equation*}
\int_M\nm{d(\chi v)}^2dv_g+ \frac{n-2}{n(n-1)} \int_M R_g (\chi v)^2dv_g\leq \frac{n-2}{n-1}\int_M r_1(\chi v)^2dv_g+\int_M \nm{d\chi}^2v^{2}dv_g.
\end{equation*}
hence
\begin{multline*}
\int_M\left(\frac{4(n-1)}{n-2}\nm{d(\chi v)}^2+ R_g (\chi v)^2\right)dv_g\leq 4\int_M\left(r_1+\frac{n-4}{4n}R_g\right)(\chi v)^2dv_g\\
+\frac{4(n-1)}{n-2}\int_M \nm{d\chi}^2v^{2}dv_g.
\end{multline*}
For $R>0$, we introduce the functions:
 \begin{equation*}
 \chi_{R}(x)=\begin{cases}
 1&\text{ on }B(x_0,R)\\
 2-\frac{d(x,x_0)}R&\text{ on }B(x_0,2R)\setminus B(x_0,R)\\
 0&\text{ on }M\setminus B(x_0,2R)\end{cases}
\end{equation*}
where $x_0\in M$ is a fixed point.

According to the Hölder inequality,
\begin{multline*}
\int_M\left(\frac{4(n-1)}{n-2}\nm{d(\chi_R v)}^2+ R_g (\chi_R v)^2\right)dv_g\leq 4 \nnm{r_1+\frac{n-4}{4n}R_g}_{L^{\frac{n}{2}}}\nnm{\chi_R v}_{L^{\frac{2n}{n-2}}}^2\\
+\epsilon(R)\nnm{d\chi_R}_{L^{2(n-1)}}^2,
\end{multline*}
where
\begin{equation*}
\epsilon(R)=\frac{4(n-1)}{n-2} \left(\int_{B(x_0,2R)\setminus B(x_0,R)}\nm{\xi}^2dv_g\right)^{\frac {n-2}{n-1}}.
\end{equation*}
Therefore
\begin{align*}
\Y(M,[g])\nnm{\chi_R v}_{L^{\frac{2n}{n-2}}}^2&\leq\int_M\left(\frac{4(n-1)}{n-2}\nm{d(\chi_R v)}^2+ R_g (\chi_R v)^2\right)dv_g\\
&\leq 4 \nnm{r_1+\frac{n-4}{4n}R_g}_{L^{\frac{n}{2}}}\nnm{\chi_R v}_{L^{\frac{2n}{n-2}}}^2 +\epsilon(R)\nnm{d\chi_R}_{L^{2(n-1)}}^2.
\end{align*}
We have
 \begin{equation*}
   \nnm{d\chi_{R}}_{L^{2(n-1)}}^2\leq  \frac 1{R^2}\vol(B(x_0,2R))^{\frac{1}{n-1}},
 \end{equation*}
and according to \cite[Theorem 1]{gallot}  (see also \cite{yauall},\cite{aubry}), when the lowest eigenvalue $Ric_-$ of the Ricci curvature is in $L^q$ for some $\frac n2<q\leq n-1$, then 
\begin{equation}\label{eq:VBallControl}
 \vol B(x_0,R)=O\left(R^{2(n-1)}\right).
\end{equation}
Since $Ric_-=-r_1+\frac{R_g}{n}$, $Ric_-$ is in $L^{n/2}\cap L^p$ for some $p>n/2$, and we have \eqref{eq:VBallControl}, hence
\begin{equation*}
 \epsilon(R)\nnm{d\chi_R}_{L^{2(n-1)}}^2\xrightarrow[R\to\infty]{} 0.
\end{equation*}
Since $v$ is in $L^{\frac{2n}{n-2}}$ and $R_g$ is in $L^{\frac n2}$, we get by Lebesgue's dominated convergence theorem that the function $v$ satisfies
\begin{align*}
\Y(M,[g])\nnm{v}_{L^{\frac{2n}{n-2}}}^2&\leq \frac{4(n-1)}{n-2}\int_M\nm{dv}^2dv_g+ \int_M R_g v^2dv_g\\
&\leq 4 \nnm{r_1+\frac{n-4}{4n}R_g}_{L^{\frac{n}{2}}}\nnm{v}_{L^{\frac{2n}{n-2}}}^2.
\end{align*}
Therefore, if $v$ doesn't vanish, then
\begin{equation}\label{eq:inegcomplete}
\frac{1}{4}\Y(M,[g]) \leq\nnm{r_1+\frac{n-4}{4n}R}_{L^{\frac{n}{2}}}\leq\nnm{r_1}_{L^{\frac{n}{2}}}+\frac{n-4}{4n}\nnm{R}_{L^{\frac{n}{2}}}.
\end{equation}

\vspace*{10pt}

If furthermore equality holds, then $v$ is a minimizer for the Yamabe functional, thus satisfies the Yamabe equation
\begin{equation*}
\frac{4(n-1)}{n-2}\Delta_g v + R_g v=\Y(M,[g])v^{\frac{n+2}{n-2}},
\end{equation*}
 and we can suppose that
\begin{equation*}
\nnm{v}_{L^{\frac{2n}{n-2}}}=1.
\end{equation*}
As $v\in C^{0,\frac{n-2}{n-1}}$, it is smooth and positive, and the metric $\tilde g=v^{\frac{4}{n-2}}g$ has constant scalar curvature equal to $\Y(M,[g])$.  

Moreover, since equality must hold in the refined Kato inequality, then according to \pref{liwang}, $M$ or one of its two-fold covering is isometric to $N\times \R$ endowed with a metric
\begin{equation*}
\hat g=\eta^2(t) h+dt^2.
\end{equation*}
If we take the new coordinate $s=\int_0^t \eta^{-1}(\tau) d\tau$, we can write that
\begin{equation*}
 \hat  g=e^{-2f(s)}(h+ds^2),
\end{equation*}
where $s$ is in $(s_-, s_+)$, with
$$s_+=\int_0^{+\infty} \frac{dt}{\eta(t)}\quad \text{and} \quad s_-=\int_0^{-\infty} \frac{dt}{\eta(t)}.$$

Since $v=e^{(n-2)f}$ is a solution of the Yamabe equation for the metric $\hat g$, the function $w=e^{\frac{n-2}{2}f}$ is a solution of the Yamabe equation for the metric $h+(ds)^2$, hence satisfies
\begin{equation}\label{ED}
-4\frac{n-1}{n-2}w''(s)+R_h w(s)=\Y(M,[g])w(s)^{\frac{n+2}{n-2}}.
\end{equation}
In particular, we see that $R_h$ only depends on $s$, hence is constant.

We can now prove that $s_+=+\infty$. Recall that
\begin{equation*}
|\xi| =\eta^{1-n}=e^{(n-1)f}=w^{2\frac{n-1}{n-2}}.
\end{equation*}
If $s_+$ is finite, then because of (\ref{limitin}), we get
$$\lim_{s\to s_+} w=0.$$
The differential equation (\ref{ED}) implies that $w'$ must have a non-zero limit when $s\to s_+$. Therefore, there exists $c>0$ such that
$$w\underset{s\to s_+}\sim c(s_+-s).$$
And since the metric $\hat g= w^{-\frac{4}{n-2}}(h+ds^2)$ is complete, we must have
\begin{equation*}
\int_0^{s^+}w(s)^{-\frac{2}{n-2}}ds=+\infty,
\end{equation*}
hence $n=4$. But according to \eqref{eq:scalarChange}, when $n=4$, the scalar curvature of $\hat g$ satisfies
\begin{equation*}
  \frac 1{w^3}R_{\hat g}=-6\left(\frac 1w\right)''+\frac 1w R_h,
\end{equation*}
thus $R_{\hat g}$ goes to $-12c^{2}$ when $s\to s_+$, and therefore is not in $L^{n/2}(M,\hat g)$. Consequently, $s_+=+\infty$, and the same argument shows that $s_-=-\infty$.

From (\ref{ED}), we deduce that there is a constant $c$ such that
$$-4\frac{n-1}{n-2}(w')^2+R_h w^2=\frac{n-2}{n}\Y(M,[g])w^{\frac{2n}{n-2}}+c.$$
Since $\lim_{s\to \pm \infty} w=0$, we must have $c=0$. Moreover, since $\Y(M,[g]$ is positive, and $w$ is a positive function we must also have $R_h>0$. Up to a change of time variable and a scaling on $\hat g$, we can suppose that
\begin{equation*}
R_h=(n-2)(n-1). 
\end{equation*}
Let $\varphi=e^{-f}=w^{-\frac{2}{n-2}}$. We obtain
 \begin{equation*}
-(\varphi')^2+\varphi^2=\frac{\Y(M,[g])}{4n(n-1)}.
\end{equation*}
Therefore, for some $s_0$, we have 
$$\varphi(s)=\sqrt{\frac{\Y(M,[g])}{4n(n-1)}} \, \cosh(s-s_0)\,.$$
\vspace*{10pt}

Conversely, if $(N^{n-1},h)$ is a closed manifold with positive scalar curvature
$$R_h=(n-2)(n-1),$$
and if
\begin{equation*}
(M,g)=\left(N^{n-1}\times\mathbb R,\, \alpha\cosh^2(t)\left(h+(dt)^2\right)\right),
\end{equation*}
then
\begin{equation*}\begin{split}
R_{g}&=(n-1)(n-4)\frac{1}{\alpha\,\cosh^4(t)} \\
\rst_g&=\rst_h+2\frac{n-2}{n}\left(h-(n-1)ds^2\right)
\end{split}\end{equation*}
We see that the lowest eigenvalue of $\rst_g$ satisfies
$$r_1(g)=\frac{1}{\alpha\cosh^4(t)}\left(r_1(h)+\frac{2(n-2)(n-1)}{n}\right),$$
and thus we obtain
\begin{equation*}
\nnm{r_1}_{L^{\frac{n}{2}}}+\frac{n-4}{4n}\nnm{R_g}_{L^{\frac{n}{2}}}=\frac{C}{4} \Y(M,[g]),
\end{equation*}
with
\begin{equation*}
  \begin{split}
C&=\frac{n(n-1)+4 r_1(h)}{\Y(M,[g])}\vol((N,h))^{\frac{2}{n}}\left(\int_{\mathbb R}\frac{dt}{\cosh^n(t)}\right)^{\frac{2}{n}}\\
&=\left(1+\frac{4 r_1(h)}{n(n-1)}\right)\left(\frac{\vol((N,h))}{\vol(\mathbb S^{n-1})}\right)^{\frac{2}{n}}\frac{\Y(\mathbb S^n)}{\Y(N\times\mathbb R,[h+dt^2])}\,.
\end{split}
\end{equation*}
According to \cite[Proposition 2.12]{AB}, we always have
\begin{equation*}
\left(\frac{\vol((N,h))}{\vol(\mathbb S^{n-1})}\right)^{\frac{2}{n}}\frac{\Y(\mathbb S^n)}{\Y(N\times\mathbb R,[h+dt^2])}\geq 1.
\end{equation*}
Hence, for $C$ to be equal to $1$, $r_1(h)$ must vanish, i.e. $h$ must be Einstein. Then, according to \pref{prop:YamabeCylin}, we have
\begin{equation*}
\left(\frac{\vol((N,h))}{\vol(\mathbb S^{n-1})}\right)^{\frac{2}{n}}\frac{\Y(\mathbb S^n)}{\Y(N\times\mathbb R,[h+dt^2])}=1,
\end{equation*}
and it follows that equality holds in \eqref{eq:degre1completeigen}.

\end{document}